\documentclass{amsart}
\usepackage{amsmath}
\usepackage{thmtools}
\usepackage{thm-restate}
\usepackage{amssymb}
\usepackage{mathtools}
\usepackage{float}
\usepackage{hyperref}

\newcommand{\Z}{\mathbb Z}
\newcommand{\sm}[4]{\begin{psmallmatrix}#1&#2\\#3&#4 \end{psmallmatrix}}
\newcommand{\set}[1]{\left\{#1\right\}}
\newcommand{\inv}{^{-1}}
\newcommand{\paren}[1]{\left(#1\right)}
\newcommand{\ceil}[1]{\left\lceil#1\right\rceil}

\newtheorem{theorem}{Theorem}
\newtheorem{lemma}[theorem]{Lemma}

\title[Coefficients of modular functions]{Congruences for coefficients of level 2 modular functions with poles at 0}
\author{Paul Jenkins}
\email{jenkins@math.byu.edu}
\author{Ryan Keck}
\email{tehrmc08@gmail.com}
\author{Eric Moss}
\email{ericbm2@byu.edu}
\thanks{This work was partially supported by a grant from the Simons Foundation ($\# 281876$ to Paul Jenkins).}
\begin{document}

\begin{abstract}
We give congruences modulo powers of 2 for the Fourier coefficients of certain level 2 modular functions with poles only at 0, answering a question posed by Andersen and the first author. The congruences involve a modulus that depends on the binary expansion of the modular form's order of vanishing at $\infty$.
\end{abstract}

\maketitle

\section{Introduction}
A modular form $f(z)$ of level $N$ and weight $k$ is a function which is holomorphic on the upper half plane, satisfies the equation
\[f\paren{\frac{az+b}{cz+d}} = (cz+d)^k f(z) \text{ for all } \sm{a}{b}{c}{d} \in \Gamma_0(N),\]
and is holomorphic at the cusps of $\Gamma_0(N).$ Letting $q=e^{2 \pi i z},$ these functions have Fourier series representations of the form $f(z) = \sum_{n = 0}^\infty a(n)q^n.$ A {\em weakly holomorphic} modular form is a modular form that is allowed to be meromorphic at the cusps. We define $M_k^\sharp (N)$ to be the space of weakly holomorphic modular forms of weight $k$ and level $N$ that are holomorphic away from the cusp at $\infty,$ and define $M_k^\flat(N)$ similarly, but for forms holomorphic away from the cusp at 0.

The coefficients of many modular forms have interesting arithmetic properties; for instance, the coefficients $c(n)$ of the $j$-invariant $j(z) = q \inv + 744 + \sum_{n=1}^\infty c(n)q^n$ appear as linear combinations of dimensions of irreducible representations of the Monster group. Also, modular form coefficients often satisfy certain congruences. Lehner \cite{lehner1,lehner2} proved that the $c(n)$ satisfy the congruence
\[c(2^a3^b5^c7^dn) \equiv 0 \pmod{2^{3a+8}3^{2b+3}5^{c+1}7^d}.\]
Many others have obtained similar congruences, in particular for forms in the space $M_0^\sharp(N).$ Kolberg \cite{kolberg1,kolberg2}, Aas \cite{aas}, and Allatt and Slater \cite{allatt} strengthened Lehner's congruences for $j(z)$, and Griffin \cite{griffin} extended Kolberg's and Aas's results to all elements of a canonical basis for $M_0^\sharp(1).$ The first author, Andersen, and Thornton \cite{andersen,thornton1} proved congruences for Fourier coefficients of canonical bases for $M_0^\sharp(p)$ with $p = 2,$ $3,$ $5,$ and $7.$ A natural question is whether similar congruences hold for coefficients of canonical bases for spaces where we allow poles at some other cusp.

Recall that the cusps of $\Gamma_0(2)$ are $0$ and $\infty.$ Taking $\eta(z) = q^{\frac{1}{24}} \prod_{n=1}^\infty \left(1 - q^n \right)$ to be the Dedekind eta function, a Hauptmodul for $\Gamma_0(2)$ is
\[\phi(z) = \paren{\frac{\eta(2z)}{\eta(z)}}^{24} =  q+24q^2+\cdots,\]
which vanishes at $\infty$ and has a pole only at 0. Note that the functions $\phi(z)^m$ for $m \geq 0$ are a basis for $M_0^\flat(2).$ Andersen and the first author used powers of $\phi(z)$ to prove congruences involving $ \psi = \frac{1}{\phi} = q\inv - 24 + \cdots \in M_0^\sharp(2)$ in \cite{andersen}, and made the following remark: ``Additionally, it appears that powers of the function $[\phi(z)]$ have Fourier coefficients with slightly weaker divisibility properties... It would be interesting to more fully understand these congruences.''  In this paper, we prove congruences for these Fourier coefficients.

Write $\phi(z)^m$ as $\sum_{n=m}^\infty a(m,n)q^n$. The main result of this paper is the following theorem.

\begin{restatable}{theorem}{congruence}
\label{congruence}
Let $n = 2^\alpha n'$ where $2 \nmid n'.$ Express the binary expansion of $m$ as $a_r\dots a_2a_1,$ and consider the rightmost $\alpha$ digits $a_\alpha\dots a_2 a_1,$ letting $a_i = 0$ for $i > r$ if $\alpha > r.$ Let $i'$ be the index of the rightmost 1, if it exists. Let
\[
\gamma(m,\alpha) =
\begin{cases}
\#\set{i\ | \ a_i = 0, i > i'}+1 &\text{if } i' \text{ exists},\\
0 &\text{otherwise.}
\end{cases}
\]
Then
\[a(m,2^\alpha n') \equiv 0 \pmod{2^{3\gamma(m,\alpha)}}.\]
\end{restatable}

\noindent
That the structure of the binary expansion of $m$ appears in the modulus of this congruence is a surprising result. We note that this congruence is not sharp. For $m = 1,$ Allatt and Slater in \cite{allatt} proved a stronger result that provides an exact congruence for many $n$.

As an example, the binary expansion for $m = 40$ is $m = \cdots 000101000.$ As we increment $\alpha,$ the $\gamma$ function gives the values in Table \ref{gammaexample}.
\begin{table}[H]
\label{gammaexample}
\begin{tabular}{|c||c|c|c|c|c|c|c|c|c|c|c|c|c|c|}
\hline
$\alpha$            & 0 & 1 & 2 & 3 & 4 & 5 & 6 & 7 & 8 & 9 & $\cdots$ & $\alpha$     & $\cdots$\\
\hline
$\gamma(40,\alpha)$ & 0 & 0 & 0 & 0 & 1 & 2 & 2 & 3 & 4 & 5 & $\cdots$ & $\alpha - 4$ & $\cdots$\\ \hline
\end{tabular}
\caption{Values of $\gamma(m,\alpha)$ for $m = 40$}
\end{table}
\noindent
Notice that once $\alpha$ surpasses 6---and the leftmost 1 in the binary expansion of $m$ occurs in the 6th place---$\gamma$ always increases by 1 as $\alpha$ increases by 1. This illustrates that $\gamma(m,\alpha)$ is unbounded for a fixed $m$.

We also prove the following result on the parity of $a(1,n).$

\begin{restatable}{theorem}{oddsquares}
\label{oddsquares}
The $n$th coefficient $a(1,n)$ of $\phi(z)$ is odd if and only if $n$ is an odd square.
\end{restatable}

Section 2 contains the machinery and definitions we use in the proof of Theorem \ref{congruence}. The proof of Theorem \ref{congruence} is in Section  3, and the proof of Theorem \ref{oddsquares} is in Section 4.

\section{Preliminary Lemmas}

The operator $U_p$ on a function $f(z)$ is given by
\[U_p f(z) = \dfrac{1}{p} \sum_{j = 0}^{p-1} f\paren{\frac{z+j}{p}}.\]
We have $U_p : M_k^!(N) \to M_k^!(N)$ if $p$ divides $N$. If $f(z)$ has the Fourier expansion $\sum^\infty_{n = n_0}a(n)q^n$, then the effect of $U_p$ is given by $U_p f(z) = \sum^\infty_{n = n_0}a(pn)q^n$.

The following result describes how $U_p$ applied to a modular function behaves under the Fricke involution. This will help us in Lemma \ref{polynomial} to write $U_2\phi^m$ as a polynomial in $\phi$.

\begin{lemma}
\label{fricke}
\cite[Theorem~4.6]{apostol} Let $p$ be prime and let $f(z)$ be a level $p$ modular function. Then
\[p(U_pf)\paren{\frac{-1}{pz}} = p(U_pf)(pz)+f\paren{\frac{-1}{p^2z}}-f(z).\]
\end{lemma}

The Fricke involution $\sm{0}{-1}{2}{0}$ swaps the cusps of $\Gamma_0(2)$, which are 0 and $\infty.$ We will use this fact in the proof of Lemma \ref{polynomial}, and the following relations between $\phi(z)$ and $\psi(z)$ will help us compute this involution.

\begin{lemma}
\label{phipsi}
\cite[Lemma~3]{andersen}
The functions $\phi(z)$ and $\psi(z)$ satisfy the relations
\begin{gather*}
\phi\paren{\frac{-1}{2z}} = 2^{-12}\psi(z),\\
\psi\paren{\frac{-1}{2z}} = 2^{12}\phi(z).
\end{gather*}
\end{lemma}

The following lemma is a special case of a result from one of Lehner's papers \cite{lehner2}. It provides a polynomial whose roots are modular forms used in the proof of Theorem \ref{betterbound}.

\begin{lemma}
\label{polyrelation}
\cite[Theorem~2]{lehner2}
There exist integers $b_j$ such that
\[U_2\phi(z) = 2 (b_1 \phi(z)+ b_2 \phi(z)^2).\]
Furthermore, let $h(z) = 2^{12}\phi(z/2)$, $g_1(z) = 2^{14}\paren{ b_1 \phi(z) + b_2 \phi(z)^2 }$, and $g_2(z) = -2^{14}b_2 \phi(z)$. Then
\[h(z)^2 - g_1(z)h(z)+g_2(z) = 0.\]
\end{lemma}

In the following lemma, we extend the result from the first part of Lemma \ref{polyrelation}, writing $U_2\phi^m$ as an integer polynomial in $\phi$. In particular, we give the least and greatest powers of the polynomial's nonzero terms.

\begin{lemma}
\label{polynomial}
For all $m \geq 1$, $U_2 \phi^m \in \Z[\phi].$ In particular,
\[U_2\phi^m = \sum_{j = \ceil{m/2}}^{2m} d(m,j)\phi^j\] where $d(m,j) \in \Z$, and $d(m, \ceil{m/2})$ and  $d(m,2m)$ are not 0.
\end{lemma}

\begin{proof}
Using Lemmas \ref{fricke} and \ref{phipsi}, we have that

\begin{align*}
U_2\phi(-1/2z)^m
&= U_2\phi(2z)^m + 2\inv\phi(-1/4z)^m-2\inv \phi(z)^m \\
&= U_2\phi(2z)^m + 2^{-1-12m} \psi(2z)^m - 2\inv \phi(z)^m\\
&= 2^{-1-12m}q^{-2m} + O(q^{-2m+2})\\
2^{1+12m}U_2\phi(-1/2z)^m &= q^{-2m} + O(q^{-2m+2}).
\end{align*}
Because $\phi(z)^m$ is holomorphic at $\infty,$ $U_2 \phi(z)^m$ is holomorphic at $\infty$. So $U_2 \phi(-1/2z)^m$ is  holomorphic at 0 and, since it starts with $q^{-2m}$, must be a polynomial of degree $2m$ in $\psi(z).$ Let $b(m,j) \in \Z$ such that
\[
2^{1+12m}U_2\phi(-1/2z)^m = \sum_{j = 0}^{2m} b(m,j)\psi(z)^j,
\]
and we note that $b(m, 2m)$ is not 0. Now replace $z$ with $-1/2z$ and use Lemma \ref{phipsi} to get
\[
2^{1+12m}U_2\phi(z)^m = \sum_{j = 0}^{2m} b(m,j)2^{12j}\phi(z)^j,
\]
which gives
\[
U_2\phi(z)^m = \sum_{j = 0}^{2m} b(m,j)2^{12(j-m)-1}\phi(z)^j.
\]
If $m$ is even, the leading term of the above sum is $q^{m/2}$, and if $m$ is odd, the leading term is $q^{(m+1)/2}$, so the sum starts with $j = \ceil{m/2}$ as desired. Notice that $b(m,j)2^{12(j-m)-1}$ is an integer because the coefficients of $\phi(z)^m$ are integers.
\end{proof}

We may repeatedly use Lemma \ref{polynomial} to write $U_2^\alpha \phi^m$ as a polynomial in $\phi$. Let
\begin{equation}
\label{fdefn}
f(\ell) = \ceil{\ell/2},\ f^0(\ell) = \ell, \text{ and } f^k(\ell) = f(f^{k-1}(\ell)).
\end{equation}
Using Lemma \ref{polynomial}, the smallest power of $q$ appearing in $U_2^\alpha \phi^m$ is $f^\alpha(m).$ Lemma \ref{binarygammma} provides a connection between $\gamma(m,\alpha)$ and the integers $f^\alpha(m).$
\begin{lemma}
\label{binarygammma}
The function $\gamma(m,\alpha)$ as defined in Theorem \ref{congruence} is equal to the number of odd integers in the list
\[m, f(m), f^2(m), \dots, f^{\alpha -1}(m). \]
\end{lemma}

\begin{proof}
Write the binary expansion of $m$ as $a_r\dots a_2 a_1$, and consider its first $\alpha$ digits, $a_\alpha \dots a_2 a_1$, where $a_ i = 0$ for $i > r$ if $\alpha > r.$ If all $a_i = 0$, then all of the integers in the list are even. Otherwise, suppose that $a_i = 0$ for $1 \leq i < i'$ and $a_{i'} =1.$ Apply $f$ repeatedly to $m$, which deletes the beginning 0s from the expansion, until $a_{i'}$ is the rightmost remaining digit; that is, $f^{i'-1}(m) = a_\alpha \dots a_{i'-1}a_{i'}.$ In particular, this integer is odd. Having reduced to the odd case, we now treat only the case where $m$ is odd.

If $m$ in the list is odd, then $a_1 = 1,$ which corresponds to the $+1$ in the definition of $\gamma(m,\alpha).$ Also, $f(m) = \ceil{m/2} = (m+1)/2.$ Applied to the binary expansion of $m$, this deletes $a_1$ and propagates a 1 leftward through the binary expansion, flipping 1s to 0s, and then terminating upon encountering the first 0 (if it exists), which changes to a 1. As in the even case, we apply $f$ repeatedly to delete the new leading 0s, producing one more odd output in the list once all the 0s have been deleted. Thus, each 0 to the left of $a_{i'}$ corresponds to one odd number in the list.
\end{proof}

\section{Proof of the Main Theorem}

Theorem \ref{congruence} will follow from the following theorem.

\begin{restatable}{theorem}{betterbound}
\label{betterbound}
Let $f(\ell)$ be as in (\ref{fdefn}). Let $\gamma(m,\alpha)$ be as in Theorem \ref{congruence}, and let $\alpha \geq 1$. Define
\[
c(m,j,\alpha) =
\begin{cases}
-1 & \text{if } f^{\alpha - 1}(m) \text{ is even and is not }2j, \\
0 & \text{otherwise}.
\end{cases}
\]
Write $U_2^\alpha \phi^m =\sum\limits_{j = f^\alpha(m)}^{2^\alpha m}d(m,j,\alpha)\phi^j.$ Then
\begin{equation}
\label{goodresult}
\nu_2(d(m,j,\alpha)) \geq 8(j - f^\alpha(m)) + 3 \gamma(m,\alpha) + c(m,j,\alpha).
\end{equation}
\end{restatable}
\noindent
The $\alpha$ of Theorem \ref{betterbound} corresponds to the $\alpha$ in $n=2^\alpha n'$ of Theorem \ref{congruence}, and because our methods use the $U_2$ operator, they do not give meaningful congruences for the case when $\alpha = 0.$ Theorem \ref{betterbound} is an improvement on the following result by Lehner \cite{lehner2}.

\begin{theorem}
\label{lehnerbound}
\cite[Equation~3.4]{lehner2}
Write $U_2 ^\alpha \phi^m$ as  $\sum d(m,j,\alpha)\phi^j \in \Z[\phi].$ Then $\nu_2(d(m,j,\alpha)) \geq 8(j-1) + 3(\alpha-m+1)+(1-m).$
\end{theorem}
\noindent
In particular, Lehner's bound sometimes only gives the trivial result that the 2-adic valuation of $d(m,j,\alpha)$ is greater than some negative integer.

We prove Theorem \ref{betterbound} by induction on $\alpha$. The base case is similar to Lemma 6 from \cite{andersen}, which gives a subring of $\Z[\phi]$ which is closed under the $U_2$ operator. The polynomials are useful because their coefficients are highly divisible by 2. Here, we employ a similar technique to prove divisibility properties of the polynomial coefficients in Lemma \ref{polynomial}. We then induct to extend the divisibility results to the polynomials that arise from repeated application of $U_2.$
\\

\noindent
{\em Proof of Theorem \ref{betterbound}.} For the base case, we let $\alpha = 1$, and seek to prove the statement
\[U_2\phi^m = \sum_{j = \ceil{m/2}}^{2m} d(m,j,1)\phi^j\]
with
\begin{equation}
\label{alpha1bound}
\nu_2(d(m,j,1)) \geq 8(j-\ceil{m/2}) + c(m,j)
\end{equation}
where
\begin{equation*}
c(m,j) =
\begin{cases}
3 & m \text{ is odd},\\
0 & m =2j,\\
-1&\text{otherwise.}
\end{cases}
\end{equation*}
The term $c(m,j)$ combines $c(m,j,\alpha)$ and $3\gamma(m,\alpha)$ for notational convenience. We prove (\ref{alpha1bound}) by induction on $m$.

We follow the proof techniques used in Lemmas 5 and 6 of \cite{andersen}. From the definition of $U_2$, we have
\[U_2\phi^m = 2\inv\paren{\phi\paren{\frac{z}{2}}^m + \phi\paren{\frac{z+1}{2}}^m} = 2^{-1-12m}\paren{h_0(z)^m+h_1(z)^m}\]
where $h_\ell(z) = 2^{12}\phi\paren{\frac{z+\ell}{2}}.$ To understand this form, we construct a polynomial whose roots are $h_0(z)$ and $h_1(z)$. Let $g_1(z) = 2^{16}\cdot 3\phi(z)+2^{24}\phi(z)^2$ and $g_2(z) = -2^{24}\phi(z).$ Then by Lemma \ref{polyrelation}, the polynomial $F(x) = x^2 -g_1(z)x+g_2(z)$ has $h_0(z)$ as a root. It also has $h_1(z)$ as a root because under $z \mapsto z+1,$ $h_0(z) \mapsto h_1(z)$ and the $g_\ell$ are fixed.

Recall Newton's identities for the sum of powers of roots of a polynomial. For a polynomial $\prod_{i = 1}^n (x - x_i),$ let $S_\ell = x_1^\ell + \cdots + x_n^\ell$ and let $g_\ell$ be the $\ell$th symmetric polynomial in the $x_1, \dots, x_n.$ Then
\[
S_\ell = g_1S_{\ell-1}-g_2S_{\ell-2}+\cdots+(-1)^{\ell+1}\ell g_\ell.
\]
We apply this to the polynomial $F(x),$ which has only two roots, to find that
\[h_0(z)^m + h_1(z)^m =S_m = g_1S_{m-1}-g_2S_{m-2}.\]
Furthermore,
\begin{equation}
\label{u2phipoly}
U_2\phi^m = 2^{-1-12m}S_m.
\end{equation}

Lastly, let $R$ be the set of polynomials of the form $d(1) \phi(z) + \sum_{n=2}^N d(n) \phi(z)^n$ where for $n \geq 2$, $\nu_2(d(n)) \geq 8(n-1).$
Now we rephrase the theorem statement in terms of $S_m$ and elements of $R$. When $m$ is odd, we wish to show that for some $r \in R$, $U_2\phi^m = 2^{-8(\ceil{m/2}-1)+3}r.$ Performing straightforward manipulations using (\ref{u2phipoly}), this is equivalent to $S_m = 2^{8(m+1)} r$ for some $r \in R.$ Similarly, when $m$ is even and is not $2j$, we wish to show that $U_2\phi^m = 2^{-8(\ceil{m/2}-1)-1}r$ for some $r \in R.$ This again reduces to showing that $S_m = 2^{8(m+1)}r$ for some $r \in R.$ If $m = 2j$, then (\ref{alpha1bound}) gives $8(j-\ceil{2j/2})+0 = 0,$ which means the polynomial has integer coefficients, which is true by Lemma \ref{polynomial}.

When $m=1$ or $2,$ we have that $S_m=2^{8(m+1)}r$ for some $r \in R$, as
\begin{gather*}
S_1 = g_1 = 2^{8(2)}(3 \phi + 2^{8}\phi^2),\\
S_2 = g_1S_1 - 2g_2 = 2^{8(3)}(2 \phi + 2^{8}3^2\phi^2+2^{17}\phi^3+2^{24}\phi^4).
\end{gather*}
Now assume the equality is true for positive integers less than $m$ with $m$ at least 3. Then for some $r_1, r_2 \in R$,
\begin{align*}
S_m &= g_1S_{m-1}-g_2S_{m-2}\\
&= (2^{16}(3\phi + 2^8 \phi^2))(2^{8m}r_1)+(2^{24}\phi)(2^{8(m-1)}r_2)\\
&= 2^{8(m+1)}[(3\cdot2^8\phi+2^{16}\phi^2)r_1+2^8\phi r_2],
\end{align*}
completing the proof where $\alpha = 1.$

Assume the theorem is true for $U_2^\alpha \phi^m =\sum\limits_{j = s}^{2^\alpha m}d(j)\phi^j,$ meaning
\begin{equation}
\label{inductivehypothesis}
\nu_2(d(j)) \geq 8(j - f^\alpha(m)) + 3 \gamma(m,\alpha) + c(m,j,\alpha).
\end{equation}
Note that $s = f^\alpha(m).$ Letting $s' = f(s)$ and $U_2 \phi^j = \sum_{i = \ceil{j/2}}^{2j} b(j,i) \phi^i,$ we define $d'(j)$ as the integers satisfying the following equation:
\begin{align}
U_2^{\alpha + 1} \phi^m &= U_2 \left( \sum^{2^{\alpha} m}_{j = s} d(j) \phi^j \right) \nonumber\\
&= \sum^{2^{\alpha} m}_{j = s} d(j) U_2 \phi^j \nonumber\\
&= \sum^{2^{\alpha} m}_{j = s} \sum^{2 j}_{i = \ceil{j/2}} d(j)b(j,i) \phi^i \nonumber\\
&= \sum^{2^{\alpha + 1} m}_{j = s'} d'(j) \phi^j \label{djdjprime}.
\end{align}
We wish to prove that
\begin{equation}
\label{inductiveconclusion}
\nu_2(d'(j)) \geq 8(j - f^{\alpha+1}(m))+3\gamma(m,\alpha+1) + c(m,j,\alpha+1).
\end{equation}
We will prove inequalities that imply (\ref{inductiveconclusion}). Observe that
\begin{flalign*}
&&
c(m,j,\alpha + 1) &=
\begin{cases}
-1 &\text{if } s \text{ is even and not } 2j,\\
0 &\text{if } s\text{ is odd or }s = 2j,
\end{cases} &&\\
\text{and} && & &&
\\
&& \gamma(m,\alpha+1)&=
\begin{cases}
\gamma(m,\alpha) &\text{if } s \text{ is even,}\\
\gamma(m,\alpha)+1 & \text{if } s \text{ is odd.}
\end{cases} &&
\end{flalign*}

\noindent
Also, $c(m,s,\alpha) = 0$ because if $f^{\alpha-1}(m)$ is even, then $s = f^{\alpha-1}(m)/2$ so $f^{\alpha - 1}(m) = 2s.$ Therefore, $\nu_2(d(s)) \geq 3\gamma(m,\alpha)$ by (\ref{inductivehypothesis}).

If $s$ is even, we will show that
\begin{equation}
\label{theineq}
\nu_2 (d'(j)) \geq \max \left\{8\left(j - s'\right) -1 + \nu_2 (d(s)), \nu_2 (d(s)) \right\},
\end{equation}
because then if $j = s'$, we have
\begin{align*}
\nu_2(d'(s')) &\geq \nu_2(d(s))\\
&\geq 8(s' - s') + 3\gamma(m,\alpha) + c(m, s', \alpha + 1),
\end{align*}
and for all $j$,
\begin{align*}
\nu_2(d'(j))&\geq 8(j - s')+3\gamma(m,\alpha)+c(m,j,\alpha+1)\\
&= 8(j - f^{\alpha+1}(m))+3\gamma(m,\alpha+1)+c(m,j,\alpha+1),
\end{align*}
so that (\ref{theineq}) implies (\ref{inductiveconclusion}). If $s$ is odd we will show that
\begin{equation}
\label{theineq2}
\nu_2 (d'(j)) \geq 8\left(j - s'\right) +3 + \nu_2 (d(s)),
\end{equation}
because then
\begin{align*}
\nu_2(d'(j)) &\geq 8(j-s')+3\gamma(m,\alpha)+3\\
&= 8(j-s')+3(\gamma(m,\alpha)+1)\\
&=8(j-f^{\alpha+1}(m))+3\gamma(m,\alpha+1)+c(m,j,\alpha+1),
\end{align*}
which is (\ref{inductiveconclusion}).

For the sake of brevity, we treat here only the case where $s$ is odd. The case where $s$ is even has a similar proof. This case breaks into subcases. We will only show the proof where $j \leq 2s$, but the other cases are $2s < j \leq 2^{\alpha-1}m$ and $2^{\alpha - 1}m < j \leq 2^{\alpha+1}m$, using the same subcases for when $s$ is even. These subcases are natural to consider because in the first range of $j$-values, the $d(s)$ term is included for computing $d'(j)$, in the second range, there are no $d(s)$ or $d({2^{\alpha}m})$ terms, and in the third range, there is a $d({2^{\alpha}m})$ term.

Let $j \leq 2s$. Using (\ref{djdjprime}), we know that $d'(j) = \sum^{2j}_{i = s}d(i) b(i,j)$ by collecting the coefficients of $\phi^j$. Let $\delta(i)$ be given by
\[\delta(i) = \nu_2(d(i)) + \nu_2(b(i, j)).\]
Let $D = \set{\delta(i) \ | \ s \leq i \leq 2j}.$ Therefore we have
\begin{align*}
\nu_2 (d'(j)) &\geq \min \set{ \nu_2(d(i)) + \nu_2(b(i,j)) \mid s \leq i \leq 2j} \\
&= \min D.
\end{align*}
We claim that $\delta(i)$ achieves its minimum with $\delta(s),$ which proves (\ref{theineq2}). For that element of $D$, we know by inequality (\ref{alpha1bound}) that
\[\delta(s) \geq \nu_2(d(s)) + 8(j - s') + 3.\]
Now suppose $i > s$. Then every element of $D$ satisfies the following inequality:
\begin{align*}
\delta(i) &= \nu_2(d(i)) + 8\left(j - \ceil{i/2}\right) + c(i,j)\\
&\geq 8\left( i - s \right) - 1 + \nu_2(d(s)) + 8\left(j - \ceil{i/2}\right) + c(i,j) \\
&\geq 8\left(s + 1 - s + j - \ceil{(s+1)/2} \right) - 2 + \nu_2(d(s)) \\
&= 8\left(j - s' \right) + 6 + \nu_2(d(s)),
\end{align*}
but this is clearly greater than $\delta(s)$. Therefore, if $j \leq 2s$ and $s$ is odd, then $\nu_2 (d'(j)) \geq 8\left(j - s'\right) +3 + \nu_2 (d(s))$. The other cases are similar.
\qed
\\

Now Theorem \ref{congruence} follows easily from Theorem \ref{betterbound}.

\congruence*

\begin{proof}
Letting $j = f^\alpha(m)$ in (\ref{goodresult}), the right hand side reduces to
\[3\gamma(m,\alpha) + c(m,f^\alpha(m),\alpha).\]
Notice that $c(m,f^\alpha(m),\alpha)=0,$ because if $f^{\alpha-1}(m)$ is even, then $f^\alpha(m)=f^{\alpha-1}(m)/2$ so $f^{\alpha - 1}(m) = 2f^\alpha(m).$ The right hand side of (\ref{goodresult}) is minimized when $j = f^\alpha(m),$ so we conclude that $\nu_2(a(m,2^\alpha n')) \geq 3 \gamma(m,\alpha).$
\end{proof}

\section{The Parity of $a(1,n)$}

Table \ref{oddcoeffs} contains all odd coefficients of $\phi(z) = \sum_{n = 1}^\infty a(1,n)q^n$ up to $n = 225.$ The table shows that, up to $n = 225,$ the coefficient $a(1,n)$ is odd if and only if $n$ is an odd square. This holds in general.

\oddsquares*

\begin{proof}
Substitute $\eta(z)$ into the definition of $\phi(z):$
\[\phi(z) = \left(\frac{\eta(2z)}{\eta(z)}\right)^{24} = \left(\frac{ q^{2/24} \prod\limits_{n=1}^{\infty} (1-q^{2n}) }{ q^{1/24} \prod\limits_{n=1}^{\infty} (1-q^{n})}\right)^{24}.\]
By recognizing that $(1-q^{2n}) = (1- q^n)(1+ q^n)$ and simplifying, it is easy to see that
\[ \phi(z) = q \prod_{n=1}^{\infty} (1+q^{n})^{24}.\]
Reducing this mod 2, the odd coefficients will be the only nonzero terms. But $\binom{24}{i}$ is odd if and only if $i = 0,8,16,24$. It follows that
\[ \phi(z) \equiv q \prod_{n=1}^{\infty} (1+q^{8n} + q^{16n} + q^{24n}) \pmod{2}.\]
Immediately, it is clear that the coefficient of $q^n$ in the Fourier expansion of $\phi(z)$ is even if $n \not \equiv 1 \pmod{8}$.

Note that the coefficient of $q^n$ in the product $\prod_{n=1}^{\infty} (1+q^{n} + q^{2n} + q^{3n})$ is odd if and only if the coefficient of $q^{8n + 1}$ is odd in the Fourier expansion of $\phi(z)$. Furthermore, this product can be interpreted as the generating function for the number of partitions of $n$ where each part is repeated at most 3 times. The $n$th coefficient of the generating function is equivalent mod 2 to $T_n$ of \cite{partitions}. Theorem 2.1 of \cite{partitions} shows that $n$ is a triangular number if and only if $T_n$ is odd. Therefore, the coefficient of $q^n$ is odd in the Fourier expansion of $\phi(z)$ if and only if
\[n = 8\frac{k(k+1)}{2} + 1 = 4k^2 + 4k + 1 = (2k+1)^2,\]
meaning that $n$ is an odd square.
\end{proof}

\begin{table}[hpbt]
\begin{tabular}{|l|l|}
\hline
$n$ & $a(1,n)$ \\ \hline
1 & 1                                                       \\ \hline
9 & 10400997                                                \\ \hline
25 & 254038914924791                                        \\ \hline
49 & 8032568516459357451913                                 \\ \hline
81 & 288274504516836871723618295721                         \\ \hline
121 & 11156646861439805613118172199024038253                \\ \hline
169 & 453988290543887189391963063089337222684846687         \\ \hline
225 & 19146547947132951990683661128349583597266368489785587 \\ \hline
\end{tabular}
\caption{All odd coefficients of $\phi(z)$ up to $n = 225.$}
\label{oddcoeffs}
\end{table}

\bibliographystyle{amsplain}

\begin{thebibliography}{10}

\bibitem{aas}
Hans-Fredrik Aas, \emph{Congruences for the coefficients of the modular
  invariant {$j(\tau )$}}, Math. Scand. \textbf{15} (1964), 64--68.
  \MR{0179138}

\bibitem{allatt}
P.~Allatt and J.~B. Slater, \emph{Congruences on some special modular forms},
  J. London Math. Soc. (2) \textbf{17} (1978), no.~3, 380--392. \MR{0491502}

\bibitem{andersen}
Nickolas Andersen and Paul Jenkins, \emph{Divisibility properties of
  coefficients of level {$p$} modular functions for genus zero primes}, Proc.
  Amer. Math. Soc. \textbf{141} (2013), no.~1, 41--53. \MR{2988709}

\bibitem{apostol}
Tom Apostol, \emph{Modular functions and dirichlet series in number theory}, 2
  ed., vol.~41, Springer-Verlag, 1990.

\bibitem{partitions}
Alex Fink, Richard Guy, and Mark Krusemeyer, \emph{Partitions with parts
  occurring at most thrice}, Contrib. Discrete Math. \textbf{3} (2008), no.~2,
  76--114. \MR{2458135}

\bibitem{griffin}
Michael Griffin, \emph{Divisibility properties of coefficients of weight 0
  weakly holomorphic modular forms}, Int. J. Number Theory \textbf{7} (2011),
  no.~4, 933--941. \MR{2812644}

\bibitem{thornton1}
Paul Jenkins and D.~J. Thornton, \emph{Congruences for coefficients of modular
  functions}, Ramanujan J. \textbf{38} (2015), no.~3, 619--628. \MR{3423017}

\bibitem{kolberg1}
O.~Kolberg, \emph{The coefficients of {$j(\tau )$} modulo powers of {$3$}},
  Arbok Univ. Bergen Mat.-Natur. Ser. \textbf{1962} (1962), no.~16, 7.
  \MR{0158061}

\bibitem{kolberg2}
\bysame, \emph{Congruences for the coefficients of the modular invariant
  {$j(\tau )$}}, Math. Scand. \textbf{10} (1962), 173--181. \MR{0143735}

\bibitem{lehner1}
Joseph Lehner, \emph{Divisibility properties of the {F}ourier coefficients of
  the modular invariant {$j(\tau)$}}, Amer. J. Math. \textbf{71} (1949),
  136--148. \MR{0027801}

\bibitem{lehner2}
\bysame, \emph{Further congruence properties of the {F}ourier coefficients of
  the modular invariant {$j(\tau)$}}, Amer. J. Math. \textbf{71} (1949),
  373--386. \MR{0027802}

\end{thebibliography}

\end{document}